\documentclass{conm-p-l}
\usepackage{amssymb,amsmath,amsthm,enumerate,amscd,graphicx,color}
\usepackage{pdfsync}
 \usepackage[colorlinks=true, pdfstartview=FitV, linkcolor=blue, citecolor=blue, urlcolor=blue]{hyperref}
\numberwithin{equation}{section}
\newtheorem{thm}{Theorem}[section]

\newtheorem{prop}[thm]{Proposition}
\newtheorem{rem}[thm]{Remark}

\newcommand{\propref}[1]{Proposition~\ref{#1}}

\newcommand{\eqnref}[1]{~(\ref{#1})}

\begin{document}


\title{Imaginary crystal bases for $U_q(\widehat{\mathfrak{sl}(2)})$-modules in category  $\mathcal O^q_{\text{red,im}}$}
\author{ Ben Cox}
\author{Vyacheslav Futorny}
\author{Kailash C. Misra}
\address{Department of Mathematics \\
University of Charleston \\
Charleston SC, USA}
\email{coxbl@cofc.edu}
\address{Department of Mathematics\\
 University of S\~ao Paulo\\
 S\~ao Paulo, Brazil}
 \email{futorny@ime.usp.br}
 \address{Department of Mathematics\\
 North Carolina State University\\
 Raleigh, NC, USA}
 \email{misra@ncsu.edu}
 \begin{abstract}  Recently we defined imaginary crystal bases for $U_q(\widehat{\mathfrak{sl}(2)})$- modules in category $\mathcal O^q_{\text{red,im}}$ and showed the existence of such bases for reduced quantized imaginary Verma modules for $U_q(\widehat{\mathfrak{sl}(2)})$. In this paper we show the existence of imaginary crystal basis for any object in the category $\mathcal O^q_{\text{red,im}}$.  
\end{abstract}
\date{}
\thanks{}


\maketitle
\section{Introduction} 

Consider the affine Lie algebra  $\widehat{\mathfrak{g}} = \widehat{\mathfrak{sl}(2)}$  with Cartan subalgebra $\widehat{\mathfrak{h}}$.
Let $\{\alpha_0 , \alpha_1\}$ be the simple roots, $\delta = \alpha_0 + \alpha_1$ the null root and $\Delta$ the set of roots for 
$\widehat{\mathfrak{g}}$ with respect to $\widehat{\mathfrak{h}}$. Then we have a natural (standard) partition of $\Delta = \Delta_+ \cup \Delta_-$ into set of positive and negative roots which is closed (i.e. $\alpha , \beta \in \Delta_+$ and $\alpha + \beta \in \Delta$ implies $\alpha + \beta \in \Delta_+$). Corresponding to this standard partition we have a standard Borel subalgebra from which we induce the standard Verma module. Let $S =  \{ \alpha_1+k\delta \ |\ k\in \mathbb Z \} \cup \{l\delta\ |\ l \in \mathbb Z_{>0} \}$. Then $\Delta = S \cup -S$ is another closed partition of the root system $\Delta$ which is not Weyl group conjugate to the standard partition. The classification of closed partitions of the root system for affine Lie algebras was obtained by Jakobsen and Kac \cite{JK,MR89m:17032}, and independently by Futorny \cite{MR1078876,MR1175820}. 
For the affine Lie algebra 
$\widehat{\mathfrak{g}}$ the partition $\Delta = S \cup -S$ is the only nonstandard closed partition which gives rise to a nonstandard Borel subalgebra. The Verma module $M(\lambda)$ with highest weight $\lambda$ induced by this nonstandard Borel subalgebra is called the imaginary Verma module for $\widehat{\mathfrak{sl}(2)}$. Unlike the standard Verma module, the imaginary Verma module $M(\lambda)$ contains both finite and infinite dimensional weight spaces.
 
For generic $q$, consider the associated quantum affine algebra $U_q(\widehat{\mathfrak{sl}(2)})$  (\cite{MR802128}, \cite{MR797001}). Lusztig \cite{MR954661} proved that the integrable highest weight modules of  $\widehat{\mathfrak{sl}(2)}$ can be deformed to those over $U_q(\widehat{\mathfrak{sl}(2)})$ in such a way that the dimensions of the weight spaces are invariant under the deformation. 
Following the framework of \cite{MR954661} and \cite{MR1341758}, it was shown in  (\cite {MR97k:17014}, \cite{MR1662112}) that the imaginary Verma modules $M(\lambda)$ can also be $q$-deformed to the quantum imaginary Verma modules $M_q(\lambda)$ in such a way that the weight multiplicities, both finite and infinite-dimensional, are preserved.

Lusztig \cite{MR1035415} from a geometric view point and Kashiwara \cite{MR1115118} from an algebraic view point 
introduced the notion of canonical bases (equivalently, global crystal bases) for standard Verma modules
$V_q(\lambda)$ and integrable highest weight modules $L_q(\lambda)$. The crystal basis (\cite{MR1090425}) can be thought of as the $q=0$ limit of the canonical basis. An important ingredient in the construction of crystal basis by Kashiwara in
 \cite{MR1115118}, is a subalgebra $\mathcal {B}_q$ of the quantum group which acts on the negative part of the quantum group  
 by left multiplication. This subalgebra $\mathcal {B}_q$, which is called the Kashiwara algebra, played an important role in the definition of the Kashiwara operators which defines the crystal basis. The algebra  $\mathcal {B}_q$ has been defined in greater generality for integrable cases in \cite{Mas09}. In \cite{CFM10} we constructed an analog of Kashiwara algebra, denoted by $\mathcal K_q$ for the imaginary Verma module $M_q(\lambda)$ for the quantum affine algebra $U_q(\widehat{\mathfrak{sl}(2)})$ by introducing certain Kashiwara-type operators. Then we proved that a certain quotient $\mathcal N_q^-$ of  $U_q(\widehat{\mathfrak{g}})$ is a simple $\mathcal K_q$-module and gave a necessary and sufficient condition for a particular quotient  $\bar{M}_q(\lambda)$ (called reduced imaginary Verma module) of  $M_q(\lambda)$ to be simple. These results were generalized to all untwisted affine Lie algebras in  \cite{CFM14}.

In \cite{CFM17} we considered a category $\mathcal O^q_{\text{red,im}}$ of $U_q(\widehat{\mathfrak{sl}(2)})$-modules and defined a crystal-like basis which we called imaginary crystal basis for modules in this category. We showed that the reduced imaginary Verma modules $\bar{M}_q(\lambda)$ are simple objects in $\mathcal O^q_{\text{red,im}}$ and any module in $\mathcal O^q_{\text{red,im}}$ is a direct sum of reduced imaginary Verma modules for $U_q(\widehat{\mathfrak{sl}(2)})$. Then we proved the existence of imaginary crystal basis for a reduced imaginary Verma module $\bar{M}_q(\lambda)$. In this paper we prove the existence of imaginary crystal basis for any object in the category $\mathcal O^q_{\text{red,im}}$.

 \section{Quantum affine algebra $U_q(\widehat{\mathfrak{sl}(2)})$}

Let $\mathbb{F}$ denote a field of characteristic zero and $q$ be an indeterminant (not a root of unity). The {\it quantum affine algebra} $U_q(\widehat{\mathfrak{sl}(2)})$ is the associative $\mathbb F(q^{1/2})$-algebra with 1 generated by
$$ E_0, E_1, F_0, F_1, K_0^{\pm 1}, K_1^{\pm 1}, D^{\pm 1} $$
with defining relations:
\begin{align*}& DD^{-1}=D^{-1}D=K_iK_i^{-1}=K_i^{-1}K_i=1, \\
& E_iF_j-F_jE_i = \delta_{ij}\frac{K_i-K_i^{-1}}{q-q^{-1}}, \\
& K_iE_iK_i^{-1}=q^2E_i, \ \ K_i F_i K_i^{-1} =q^{-2}F_i, \\
& K_i E_jK_i^{-1} = q^{-2}E_j, \ \
K_i F_jK_i^{-1} = q^2F_j, \quad i\neq j, \\
& K_iK_j-K_jK_i = 0, \ \ K_iD-DK_i=0, \\
& DE_iD^{-1}=q^{\delta_{i,0}} E_i, \ \
DF_iD^{-1}=q^{-\delta_{i,0}} F_i, \\
& E_i^3E_j-[3]E_i^2E_jE_i+[3]E_iE_jE_i^2-E_jE_i^3 =0, \quad i\neq j, \\
& F_i^3F_j-[3]F_i^2F_jF_i+[3]F_iF_jF_i^2-F_jF_i^3 = 0, \quad i\neq j, \\
\end{align*}where, $[n] = \frac{q^n-q^{-n}}{q-q^{-1}}$ and $i, j \in \{0,1\}$.

There is an alternative realization for $U_q(\widehat{\mathfrak{sl}(2)})$,
due to Drinfeld
\cite{MR802128}, which we need.  Let
$U_q$ be the associative algebra with $1$ over $\mathbb F(q^{1/2})$
generated by the
elements $x^{\pm }_k$ ($k\in \mathbb Z$), $h_l$ ($l \in \mathbb Z
\setminus \{0\}$), $K^{\pm 1}$,
$D^{\pm 1}$, and $\gamma^{\pm \frac12}$ with the following defining
relations:
\begin{align}
DD^{-1}=D^{-1}D&=KK^{-1}=K^{-1}K=1,  \\
[\gamma^{\pm \frac 12},u] &= 0 \quad \forall u \in U, \\
[h_k,h_l] &= \delta_{k+l,0} \frac{[2k]}{k} \frac{\gamma^k -
\gamma^{-k}}{q-q^{-1}},  \\
[h_k,K]&=0,\quad [D,K]=0,  \\
Dh_kD^{-1}&=q^k h_k,  \\
Dx^{\pm}_kD^{-1}&=q^{ k}x^{\pm}_k,  \\
Kx^{\pm}_kK^{-1} &= q^{\pm 2}x^{\pm}_k,   \\
[h_k,x^{\pm}_l]&= \pm \frac{[2k]}{k}\gamma^{\mp \frac{|k|}{2}}x^{\pm}_{k+l}, \label{axcommutator}  \\
    x^{\pm}_{k+1}x^{\pm}_l &- q^{\pm 2}
x^{\pm}_lx^{\pm}_{k+1}\label{Serre}   \\
&= q^{\pm 2}x^{\pm}_kx^{\pm}_{l+1}
    - x^{\pm}_{l+1}x^{\pm}_k,\notag \\
[x^+_k,x^-_l]&=
    \frac{1}{q-q^{-1}}\left( \gamma^{\frac{k-l}{2}}\psi(k+l) -
    \gamma^{\frac{l-k}{2}}\phi(k+l)\right), \label{xcommutator}   \\
\text{where  }
\sum_{k=0}^{\infty}\psi(k)z^{-k} &= K \exp\left(
(q-q^{-1})\sum_{k=1}^{\infty} h_kz^{-k}\right),\\
\sum_{k=0}^{\infty}
\phi(-k)z^k&= K^{-1} \exp\left( - (q-q^{-1})\sum_{k=1}^{\infty}
h_{-k}z^k\right).
\end{align}
The isomorphism between  $U_q(\widehat{\mathfrak{sl}(2)})$
and $U_q$ is given by: 
\begin{align*}E_0 &\mapsto x^-_1K^{-1}, \ \ F_0 \mapsto Kx^+_{-1}, \\
E_1 &\mapsto x^+_0, \ \ F_1 \mapsto x^-_0, \\
K_0 &\mapsto \gamma K^{-1}, \ \ K_1 \mapsto K, \ \ D \mapsto D.
\end{align*}
If one uses the formal sums
\begin{equation}
\phi(u)=\sum_{p\in\mathbb Z} \phi(p)u^{-p},\enspace \psi(u)=\sum_{p\in\mathbb Z}\psi(p)u^{-p},\enspace
x^{\pm }(u)=\sum_{p\in\mathbb Z} x^\pm_pu^{-p}
\end{equation}
Drinfeld's relations (2.3), (2.8)-(2.10) can be written as
\begin{gather}
[\phi(u),\phi(v)]=0=[\psi(u),\psi(v)] \\
\phi(u)x^\pm (v)\phi(u)^{-1}=g(uv^{-1}\gamma^{\mp 1/2})^{\pm 1}x^\pm (v)\label{phix} \\
\psi(u)x^\pm (v)\psi(u)^{-1}=g(vu^{-1}\gamma^{\mp 1/2})^{\mp 1}x^\pm (v)\label{psix} \\
(u-q^{\pm2} v)x^\pm (u)x^\pm (v)=(q^{\pm 2}u-v)x^\pm(v)x^\pm(u) \\
[x^+(u),x^-(v)]=(q-q^{-1})^{-1}(\delta(u/v\gamma)\psi(v\gamma^{1/2})-\delta(u\gamma/v)\phi(u\gamma^{1/2}))\label{xx}
\end{gather}
where
$g(t)=g_q(t)=\sum_{k\geq 0}g(r)t^{k}$ is the Taylor series at $t=0$ of the function $(q^2t-1)/(t-q^2)$ and $\delta(z)=\sum_{k\in\mathbb Z}z^{k}$ is the formal Dirac delta function.
\begin{rem}
Writing $g(t)=g_q(t)=\sum_{r\geq 0}g(r)t^r$ we have
\begin{equation}\label{grcomp}
g(r)=g_q(r)=g_{q^{-1}}(r)=\begin{cases} q^{2}&\text{if}\quad r=0 \\ 
(1-q^{-4})q^{2(r+1)}=(q^{4}-1)q^{2(r-1)},&\text{if}\quad r>0.
\end{cases}
\end{equation}
\end{rem}

Considering Serre's relation  with $k=l$, we get
\begin{equation}\label{Serre1}
x^{-}_kx^{-}_{k+1}=q^2x^{-}_{k+1}x^{-}_k  
\end{equation}
The product on the right side is in the correct order for a basis  element.
If $k+1>l$ and $k\neq l$ in \eqnref{Serre}, then $k+1>  l+1$ so that $k\geq l+1$, and thus we can write
\begin{equation}\label{Serre2}
 x^{-}_lx^{-}_{k+1}=q^2x^{-}_{k+1}x^{-}_l   - x^{-}_kx^{-}_{l+1} +q^2 x^{-}_{l+1}x^{-}_k
\end{equation}
and then after repeating the above identity, we will eventually arrive at sums of terms that are in the correct order.    This is the opposite ordering of monomials as we had previously.

\section{$\Omega$-operators and the Kashiwara algebra $\mathcal K_q$}

Let $\mathbb N^{\mathbb N^*}$ denote the set of all functions from $\{k\delta\,|\, k\in\mathbb N^*\}$ to $\mathbb N$ with finite support.   Then we write
$$
h^+=h^{(s_k)}_+:=h_{r_1}^{s_1}\cdots h_{r_k}^{s_l},\quad h^-:=h^{(s_k)}_-=
h_{-r_1}^{s_1}\cdots h_{-r_k}^{s_l}
$$
for $f=(s_m)\in\mathbb N^{\mathbb N^*}$ where $f(r_m)=s_m$ and $f(t)=0$ for $t\neq r_i, m\leq i\leq  l$.

Let $ \mathcal N_q^-$, be the subalgebra of  $U_q(\widehat{\mathfrak{sl}(2)})$ generated by $\gamma^{\pm1/2}$, and $x^-_n$, $n\in\mathbb Z$. Consider $x^-(v)=\sum_nx^-_nv^{-n}$ as a formal power series of left multiplication operators $x^-_n:\mathcal N_q^-\to \mathcal N_q^-$. 

As in \cite{CFM10}, for fixed $k$ we set 
\begin{align*}
\bar P&=x^-(v_1)\cdots x^-(v_k) \\
\bar P_l&=x^-(v_{1})\cdots x^-(v_{l-1})x^-(v_{l+1})\cdots x^-(v_k),
\end{align*}
and
$$
G_l=G_l^{1/q}:=\prod_{j=1}^{l-1}g_{q^{-1}}(v_j/v_l),\quad G_l^{q}=\prod_{j=1}^{l-1}g(v_l /v_j)
$$
where $G_1:=1$.
Then we define a collection of operators $\Omega_\psi(k),\Omega_\phi(k):\mathcal N_q^-\to \mathcal N_q^-$, $k\in\mathbb Z$, in terms of the generating functions as follows.
$$
\Omega_\psi(u)=\sum_{l\in\mathbb Z}\Omega_\psi(l)u^{-l},\quad \Omega_\phi(u)=\sum_{l\in\mathbb Z}\Omega_\phi(l)u^{-l}
$$
where 
\begin{align}\label{definingomegapsi}
\Omega_\psi(u)(\bar P):&=\gamma^{m} \sum_{l=1}^kG_l
  \bar P_l \delta(u/v_l\gamma) \\
   \Omega_\phi(u)(\bar P):&=\gamma^{m}\sum_{l=1}^kG_l^q  \bar P_l\delta(u\gamma/v_l).\label{definingomegaphi}
\end{align}

Note that $\Omega_{\psi}(u)(1)=\Omega_{\phi}(u)(1)=0$.  More generally, let us write
$$
\bar P=x^-(v_1)\cdots x^-(v_k)=\sum_{n\in\mathbb Z}\sum_{n_1,n_2,\dots,n_k\in\mathbb Z\atop n_1+\cdots +n_k=n}x^-_{n_1}\cdots x^-_{n_k}v_1^{-n_1}\cdots v_k^{-n_k}
$$
Then
\begin{align*}
\psi&(u\gamma^{-1/2})\Omega_\psi(u)(\bar P)\\
&=\sum_{k\geq 0}\sum_{p\in\mathbb Z}\sum_{n_i\in\mathbb Z } \gamma^{k/2}\psi(k)\Omega_\psi(p)(x^-_{n_1}\cdots x^-_{n_k})v_1^{-n_1}\cdots v_k^{-n_k}u^{-k-p}  \\
&=\sum_{n_i\in\mathbb Z } \sum_{m\in\mathbb Z}\sum_{k\geq 0}\gamma^{k/2}\psi(k)\Omega_\psi(m-k)(x^-_{n_1}\cdots x^-_{n_k})v_1^{-n_1}\cdots v_k^{-n_k}u^{-m}
\end{align*}
while
\begin{equation*}
[x^+(u),\bar P]=\sum_{m\in\mathbb Z}\sum_{n_1,n_2,\dots,n_k\in\mathbb Z } [x^+_m,x^-_{n_1}\cdots x^-_{n_k}]v_1^{-n_1}\cdots v_k^{-n_k}u^{-m}.
\end{equation*}
Thus for a fixed $m$ and $k$-tuple $(n_1,\dots,n_k)$ the sum
$$
\sum_{k\geq 0}\gamma^{k/2}\psi(k)\Omega_\psi(m-k)(x^-_{n_1}\cdots x^-_{n_k})
$$
must be finite.  Hence
\begin{equation}\label{omegalocalfin}
\Omega_\psi(m-k)(x^-_{n_1}\cdots x^-_{n_k})=0,
\end{equation}
 for $k$ sufficiently large.

\begin{prop}[\cite{CFM10}, Prop. 4.0.3] \label{commutatorprop} We have
\begin{align}
\Omega_\psi(u)x^-(v)&=\delta(v\gamma/u)+g_{q^{-1}}(v\gamma/u)x^-(v)\Omega_\psi(u),
\label{omegapsi}\\
  \Omega_\phi(u)x^-(v)&=\delta(u\gamma/v)+g(u\gamma/v)x^-(v)\Omega_\phi(u)\label{omegaphi}  \\
(q^2u_1-u_2)\Omega_\psi(u_1)\Omega_\psi(u_2)&=(u_1-q^2u_2)\Omega_\psi(u_2)\Omega_\psi(u_1) \label{psipsi} \\
(q^2u_1-u_2)\Omega_\phi(u_1)\Omega_\phi(u_2)&=(u_1-q^2u_2)\Omega_\phi(u_2)\Omega_\phi(u_1)   \label{phiphi} \\
(q^2\gamma^2u_1-u_2)\Omega_\phi(u_1)\Omega_\psi(u_2)&=(\gamma^2u_1-q^2u_2)\Omega_\psi(u_2)\Omega_\phi(u_1).\label{omegaphipsi}
\end{align}
\end{prop}

The identities in \propref{commutatorprop} can be rewritten as
\begin{align}
(q^2v\gamma- u)\Omega_\psi(u)x^-(v)&=(q^2v\gamma- u)\delta(v\gamma/u)+(q^2v\gamma -u)x^-(v)\Omega_\psi(u),
\label{omegapsi2}\\
(q^2v- u\gamma)  \Omega_\phi(u)x^-(v)&=(q^2v- u\gamma)\delta(v/u\gamma)+( v-q^2u\gamma)x^-(v)\Omega_\phi(u)\label{omegaphi3}
\end{align}
which may be written out in terms of components as
\begin{align}
&q^2\gamma\Omega_\psi(m)x^-(n+1)- \Omega_\psi(m+1)x^-_n\label{omegapsi4} \\
&\quad =(q^2\gamma-1)\delta_{m,-n-1}+ \gamma x^-_{n+1}\Omega_\psi(m)-q^2x^-_n\Omega_\psi(m+1),
\notag \\
&  q^2\Omega_\phi(m)x^-(n+1)-  \gamma\Omega_\phi(m+1)x^-(n) \label{omegaphi5} \\
&\qquad =(q^2- \gamma)\delta_{m,-n-1}+ x^-(n+1)\Omega_\psi(m)-q^2\gamma x^-(n)\Omega_\psi(m+1).\notag 
\end{align}

  We also have by \eqnref{omegaphipsi}
  \begin{equation}\label{omegaphipsi2}
    \Omega_\psi(k)\Omega_\phi(m)= \sum_{r\geq 0}g (r)\gamma^{2r}\Omega_\phi(r+m)\Omega_\psi(k-r),
\end{equation}
as operators on $\mathcal N_q^-$.

  We can also write \eqnref{omegapsi} in terms of components and as operators on $\mathcal N_q^-$
\begin{equation}\label{omegapsi6}
    \Omega_\psi(k)x^-(m)=\delta_{k,-m}\gamma^{k}+\sum_{r\geq 0}g_{q^{-1}}(r)x^-(m+r)\Omega_\psi(k-r)\gamma^{r}.
\end{equation}
    The sum on the right hand side turns into a finite sum when applied to an element in $\mathcal N_q^-$, due to \eqnref{omegalocalfin}.

\begin{prop}\cite{CFM10} \label{form}
There is a unique nondegenerate symmetric bilinear form $(\enspace, \enspace)$ defined on $\mathcal N^-_q$ satisfying
$$
(x^-_ma,b)=(a,\Omega_\psi(-m)b),\quad (1,1)=1.
$$
\end{prop}

  For $\mathbf m=(m_1,\dots, m_n)$ set
$$
x_{\mathbf m}=x_{m_1}^-\cdots x_{m_n}^-
$$
and define the length of such a Poincare-Birkhoff-Witt basis element to be $|\mathbf m|=n$.
  \begin{prop}\cite{CFM17}\label{symmetricform}  For  $\mathbf m=(m_1,\dots, m_n)\in\mathbb Z^n$, and $\mathbf k=(k_1,\dots, k_l)\in \mathbb Z^l$, if $n>l$, then
  \begin{equation} \label{orthonormal} 
 (x_{\mathbf m},x_{\mathbf k}) =0 . 
 \end{equation}
 On the other hand if $n=l$ with 
\begin{gather*}
  m_1\geq m_2\geq \cdots \geq m_n,\quad k_1\geq k_2\geq \cdots \geq k_n, \\
  \sum_{i=1}^nm_i=  \sum_{i=1}^nk_i
\end{gather*}
   we have
 \begin{equation} \label{orthonormal} 
 (x_{\mathbf m},x_{\mathbf k}) \equiv\delta_{\mathbf m,\mathbf k}\mod q^2\mathbb Z[\![q]\!].
\end{equation}

  \end{prop}

The Kashiwara algebra $\mathcal K_q$ is defined to be the $\mathbb F(q^{1/2})$-subalgebra of $\text{End}\,( \mathcal N_q^-)$ generated by $\Omega_\psi(m),x^-_n,\gamma^{\pm 1/2}$, $m,n\in\mathbb Z$, $\gamma^{\pm 1/2}$.  Then the $\gamma^{\pm 1/2}$ are central and the following relations (which are implied by \eqnref{omegapsi6}) are satisfied
\begin{align}
q^2\gamma\Omega_\psi(m)&x^-_{n+1}-  \Omega_\psi(m+1)x^-_n \\
&=(q^2\gamma-1)\delta_{m,-n-1}+ \gamma x^-_{n+1}\Omega_\psi(m)-q^2x^-_n\Omega_\psi(m+1) \notag \\\notag \\
q^2 \Omega_\psi(k+1)&\Omega_\psi(l) -
\Omega_\psi(l)\Omega_\psi(k+1)  =  \Omega_\psi(k)\Omega_\psi(l+1)
    - q^2\Omega_\psi(l+1)\Omega_\psi(k)\label{omegapsi3}
\end{align}
\begin{equation}
 x^{-}_lx^{-}_{k+1}-q^2x^{-}_{k+1}x^{-}_l  =q^2 x^{-}_{l+1}x^{-}_k - x^{-}_kx^{-}_{l+1} \label{xminusreln}
\end{equation}
 together with
\[
\gamma^{1/2}\gamma^{-1/2}=1=\gamma^{-1/2}\gamma^{1/2}.
\]

Then $\mathcal N_q^-$ is a simple $\mathcal K_q$-module (\cite{CFM10}, Theorem 7.0.14).

\section{Quantized Imaginary Verma Modules and Category $\mathcal O^q_{\text{red,im}}$}

We begin by recalling some basic facts for the affine
Kac-Moody algebra $ \widehat{\mathfrak{sl}(2)}$ over field $\mathbb F$  with
generalized Cartan matrix $A=(a_{ij})_{0\le i,j \le 1} = \begin{pmatrix}2&-2 \\ -2&\
2 \\ \end{pmatrix}$ \\
and its imaginary Verma modules. We use notations in  \cite{K}.

$$
\mathfrak{g}= \widehat{\mathfrak{sl}(2)}
= {\mathfrak sl}(2) \otimes \mathbb F[t,t^{-1}] \oplus \mathbb F c \oplus
\mathbb F d
$$
with Lie bracket relations
\begin{align*}[x\otimes t^n,y\otimes t^m]&= [x,y] \otimes t^{n+m}
+ n \delta_{n+m,0}(x,y)c, \\
[x \otimes t^n,c] = 0 = [d,c] , \quad & \quad [d,x \otimes t^n] = nx \otimes t^n,
\end{align*}for $x,y \in {\mathfrak sl}(2)$, $n,m \in \mathbb Z$, where $(\ , \ )$ denotes the
trace form on ${\mathfrak sl}(2)$.
For $x \in {\mathfrak sl(2)}$ and $n \in \mathbb Z$, we write $x(n)$ for
$x \otimes t^n$. We consider the usual basis for ${\mathfrak sl}(2)$:

$$
\left\{e=\begin{pmatrix} 0 & 1\\ 0 & 0 \end{pmatrix},\quad f=\begin{pmatrix} 0 & 0\\ 1 & 0 \end{pmatrix},\quad h=\begin{pmatrix} 1 & 0\\ 0 & -1 \end{pmatrix}\right\}.
$$
Then $\{e_0=f(1), e_1=e(0), f_0=e(-1), f_1=f(0), h_0=-h(0)+c, h_1=h(0), d\}$ generate $\mathfrak{g}$ and $\mathfrak{h} = {\rm span}\{h_0, h_1, d\}$ is the Cartan subalgebra.

Let $\Delta$ denote the root system of  $ \widehat{\mathfrak{sl}(2)}$,
$\{ \alpha_0, \alpha_1\}$ the simple roots, $\delta = \alpha_0 + \alpha_1$
the minimal imaginary root , $\{\Lambda_0, \Lambda_1\}$ fundamental weights and $P = \mathbb Z \Lambda_0 \oplus \mathbb Z \Lambda_1 \oplus \mathbb{Z}\delta$  the weight lattice.  Then
$$
\Delta = \{ \pm \alpha_1 + n\delta\ |\ n \in \mathbb Z\} \cup \{ k\delta\ |\ k \in \mathbb Z
\setminus \{ 0 \} \}.
$$

A subset $S$ of the root system $\Delta$ is called {\it closed}
if $\alpha, \beta \in S$ and
$\alpha+\beta \in \Delta$
implies $\alpha+\beta \in S$.  The subset $S$ is called a {\it closed
partition } of the roots if $S$ is closed,
$S \cap(-S) = \emptyset$, and $S\cup -S = \Delta$ \cite{JK},\cite{MR89m:17032},\cite{MR1078876},\cite{MR1175820}.
The set
$$
S= \{ \alpha_1+k\delta \ |\ k\in \mathbb Z \} \cup \{l\delta\ |\ l \in \mathbb Z_{>0} \}
$$
is a closed partition of $\Delta$ and is $W\times
\{\pm{1}\}$-inequivalent to the standard
partition of the root system into positive and negative roots \cite{MR95a:17030}.

Let ${\mathfrak g}_{\pm}^{(S)}=\sum_{\alpha \in S}
\hat{\mathfrak g}_{\pm \alpha}$.  We have
that ${\mathfrak g}_+^{(S)}$ is the
subalgebra generated by $e(k)$ $(k \in \mathbb Z)$ and $h(l)$ $(l\in \mathbb Z_{>0})$
and ${\mathfrak g}_-^{(S)}$ is the subalgebra generated by $f(k)$ $(k \in \mathbb Z)$ and
$h(-l)$ $(l\in \mathbb Z_{>0})$.  Since $S$ is a partition of the root system,
the algebra has a direct sum decomposition:
$$
{\mathfrak g}={\mathfrak g}_{-}^{(S)} \oplus {\mathfrak h} \oplus {\mathfrak g}_{+}^{(S)}.
$$
Let $U({\mathfrak g}_{\pm}^{(S)})$ be the universal enveloping algebra of
${\mathfrak g}_{\pm}^{(S)}$. Then, by the PBW theorem, we have
$$
U({\mathfrak g}) \cong U({\mathfrak g}_{-}^{(S)}) \otimes U({\mathfrak h})\otimes U({\mathfrak g}_{+}^{(S)}),
$$
where $U({\mathfrak g}_{+}^{(S)})$ is generated by $ e(k)$ $(k\in \mathbb Z)$, $h(l)$
$(l\in \mathbb Z_{>0})$,
$U({\mathfrak g}_{-}^{(S)})$ is generated by $f(k)$ $(k\in \mathbb Z)$, $h(-l)$ $(l\in
\mathbb Z_{>0})$ and $U({\mathfrak h})$,
the universal enveloping algebra of ${\mathfrak h}$.
A $U({\mathfrak g})$-module $V$ is called a {\it weight} module if
$V=\oplus_{\mu \in P} V_{\mu}$, where
$$
V_{\mu}=\{ v \in V\ |\ h\cdot v=\mu(h)v, c\cdot v=\mu(c)v,
d\cdot v = \mu(d)v \}.
$$
Any submodule of a weight module is a weight module.
A $U({\mathfrak g})$-module $V$ is
called an {\it $S$-highest weight module}
with highest weight $\lambda \in P$ if there is a
non-zero $v_{\lambda} \in V$ such that
(i) $u^+ \cdot v_{\lambda} = 0$ for all $u^+\in U({\mathfrak g}_{+}^{(S)})
\setminus \mathbb F^*$, (ii) $h\cdot v_{\lambda}=\lambda(h)v_{\lambda}$, $c\cdot v_{\lambda}
= \lambda(c)v_{\lambda}$, $d\cdot v_{\lambda}  = \lambda(d)v_{\lambda}$,
(iii) $V=U(\hat{\mathfrak g})\cdot v_{\lambda} = U({\mathfrak g}_{-}^{(S)}) \cdot v_{\lambda}$.
An $S$-highest weight module is a weight module.

For $\lambda \in P$, let $I_S(\lambda)$ denote the ideal of $U(\mathfrak g)$
generated by
$e(k)$ $(k\in \mathbb Z)$, $h(l)$ $(l>0)$, $h-\lambda(h) 1$,
$c-\lambda(c) 1$, $d-\lambda(d) 1$.
Then  $M(\lambda) = U(\mathfrak g)/I_S(\lambda)$ is the {\it
imaginary Verma module} of
$\mathfrak g$ with highest weight $\lambda$.
Imaginary Verma modules have many structural features similar to those of
standard Verma modules,
with the exception of the infinite-dimensional weight spaces \cite{MR95a:17030}.
It was shown in  \cite {MR97k:17014} that the imaginary Verma modules $M(\lambda)$ for $\mathfrak g$ can be $q$-deformed to the quantum imaginary Verma module $M_q(\lambda)$ 
for $U_q(\mathfrak g)$ in such a way that the weight multiplicities, both finite and infinite-dimensional, are preserved. Indeed for $\lambda \in P$ the quantum imaginary Verma module $M_q(\lambda)$ can be described as follows.
Denote by $I^q(\lambda)$ the ideal of
$U_q$ generated by $x^+(k)$, $k\in
\mathbb Z$, $a(l), l>0$, $K^{\pm 1}-q^{\lambda(h)}1$, $\gamma^{\pm
\frac{1}{2}}-q^{\pm \frac{1}{2}\lambda(c)}1$ and $D^{\pm 1}-q^{\pm
\lambda(d)}1$. Then
$$M_q(\lambda)=U_q/I^q(\lambda).$$

\begin{thm}[\cite{MR97k:17014}, Theorem 3.6]  The imaginary Verma module $M_q(\lambda)$ is simple
 if and only if $\lambda(c)\neq 0$.
\end{thm}

Suppose now that $\lambda(c)=0$. Then $\gamma^{\pm \frac{1}{2}}$ acts on $M_q(\lambda)$ by $1$.  Consider the ideal
$J^q(\lambda)$ of  $U_q$ generated by $I^q(\lambda)$ and  $a(l)$ for
all $l$. Denote
$$\bar{M}_q(\lambda)=U_q/J^q(\lambda).$$
Then $\bar{M}_q(\lambda)$ is a homomorphic image of
 $M_q(\lambda)$ which is called the \emph{reduced quantized imaginary Verma
 module}. The module $\bar{M}_q(\lambda)$ has a $P$-gradation:
 $$\bar{M}_q(\lambda)=\sum_{\xi\in P}\bar{M}_q(\lambda)_{\xi}.$$
 
 \begin{thm}[\cite{CFM10}]
Let $\lambda\in  P$ be such that $\lambda(c)=0$. Then the module $\bar{M}_q(\lambda)$
is simple if and only if $\lambda(h)\neq 0$.
\end{thm}

Consider the set  $\mathfrak h^*_{red}:=\{\lambda\in\mathfrak h^*\,|\, \lambda(c)=0,\lambda (h)\neq 0\}$.
Let $G_q$ be the quantized Heisenberg subalgebra generated by $h_k, k\in \mathbb Z\setminus \{0\}$ and $\gamma$. We say that a nonzero $U_q({\mathfrak g})$-module $V$ is $G_q$-compatible if $V$ has a decomposition $V=TF(V)\oplus T(V)$ into a sum of nonzero $G_q$-submodules such that $G_q$ is bijective on $TF(V)$ (that any nonzero element $g\in G_q$ is a bijection on $TF(V)$) and $TF(V)$ has no nonzero $U_q(\mathfrak g)$-submodule, and $G_q\cdot T(V)=0$.

The category $\mathcal O^q_{\text{red,im}}$ has as objects $U_q({\mathfrak  g})$-modules $M$ such that 
\begin{enumerate}
\item $$
M=\bigoplus_{\nu\in\mathfrak h^*_{red}}M_\nu,\quad\text{ where }\quad M_\nu=\{m\in M\,|\, Km=K^{\nu(h)}m,\enspace Dm=q^{\nu(d)}m\},
$$
\item $x^+_n$, $n\in\mathbb Z$ act locally nilpotently,
\item $M$ is $G_q$-compatible.
\end{enumerate}
The morphisms in this category are the $U_q({\mathfrak g})$-module homomorphisms. Then  for $M \in \mathcal O^q_{\text{red,im}}$
there exists $\lambda_i\in \mathfrak h^*_{red}$, $i\in I$,  with $M\cong \bigoplus_{i\in I}\bar M_q(\lambda_i)$ (\cite{CFM17}, Theorem 6.0.4).

For $M\in\mathcal O^q_{\text{red,im}}$, we can write $M=\oplus_i\bar M_q(\lambda_i)$ with $\bar M_q(\lambda_i)=\oplus \mathbb F(q^{1/2})x^-_{n_1}\cdots x^-_{n_k}v_{\lambda_i}$.  We define $\tilde\Omega_\psi(m)$ and $\tilde x_m^-$ on each $\bar M_q(\lambda_i)$ as in \eqnref{definingomegapsi}: 
\begin{align}
\tilde\Omega_\psi(m)(x^-_{n_1}\cdots x^-_{n_k}v_{\lambda_i})&:=\Omega_\psi(m)(x^-_{n_1}\cdots x^-_{n_k})v_{\lambda_i} \\
\tilde x_m^-(x^-_{n_1} \cdots x^-_{n_k}v_{\lambda_i})&:=  x_m^-x^-_{n_1}\cdots x^-_{n_k}v_{\lambda_i}.
\end{align}
Hence the following result follows.
  \begin{thm}\cite{CFM17} The operators $\tilde\Omega_\psi(m)$ and $\tilde x_m^-$ are well defined on objects in the category $\mathcal O^q_{\text{red,im}}$. Moreover on each summand $\bar M_q(\lambda_i)\cong \mathcal N_q^-$ they agree with the $\Omega_\psi(m)$ respectively left multiplication by $x_m^-$ defined as in \eqnref{definingomegapsi}. 
\end{thm}

\begin{thm}\label{commute}
 The operators $\tilde\Omega_\psi(m)$ and $\tilde x_m^-$ commute with $U_q(\mathfrak g)$-module homomorphisms.
\end{thm}
\begin{proof}
It is enough to prove the statement for a $U_q(\mathfrak g)$-module homomorphism $\nu:\bar M_q(\lambda)\to M$ from a highest weight module $\bar M_q(\lambda)$ to any module $M \in \mathcal O_{\text{ref},\text{im}}^q$ for some $\lambda \in \mathfrak h^*_{\text{red}}$.    Now since $\bar M_q(\lambda)$ is simple, its image under $\nu$ is isomorphic to $\bar M_q(\lambda)$.  

Consider now a basis element of $\bar M_q(\lambda)$ of the form $x^-_{n_1}\cdots x^-_{n_k}v_{\lambda}$.  Then
\begin{align*}
\nu\left(\tilde \Omega_\psi(m)(x^-_{n_1}\cdots x^-_{n_k}v_{\lambda})\right)&=\nu\left(\Omega_\psi(m)(x^-_{n_1}\cdots x^-_{n_k})v_{\lambda}\right) \\
&=\Omega_\psi(m)(x^-_{n_1}\cdots x^-_{n_k})\nu\left( v_{\lambda} \right)\\
&=\tilde\Omega_\psi(m)\left(x^-_{n_1}\cdots x^-_{n_k}\nu(v_{\lambda})\right)  \\
&=\tilde\Omega_\psi(m)\nu\left(x^-_{n_1}\cdots x^-_{n_k}v_{\lambda}\right)
\end{align*}
and
\begin{align*}
\nu\left(\tilde x_m^-(x^-_{n_1}\cdots x^-_{n_k})v_{\lambda}\right)&=\nu\left( x_m^-(x^-_{n_1}\cdots x^-_{n_k})v_{\lambda}\right) \\
&= x_m^-(x^-_{n_1}\cdots x^-_{n_k})\nu\left( v_{\lambda} \right)\\
&=\tilde x_m^-\left(x^-_{n_1}\cdots x^-_{n_k}\nu(v_{\lambda})\right) \\
&=\tilde x_m^-\nu\left(x^-_{n_1}\cdots x^-_{n_k}v_{\lambda}\right) .
\end{align*}

\end{proof}

\section{Imaginary crystal lattice and imaginary crystal basis}  
We first recall the definition of imaginary crystal lattice and imaginary crystal basis for modules in category  $\mathcal O^q_{\text{red,im}}$. Let $\mathbb A_0$   to be the ring of rational functions in $q^{1/2}$ with coefficients in a field $\mathbb F$ of characteristic zero, regular at $0$.  
 Let $\mathbb A= \mathbb F[q^{1/2},q^{-1/2}, \frac{1}{[n]_{q}}, n>1]$, and $\pi =\{-k\alpha_1+m\delta\,|\,k>0,m\in\mathbb Z\}\cup\{0\}$.    Let $M\in\mathcal O^q_{\text{red,im}}$. We call a free $\mathbb A_0$-submodule $\mathcal L$ of $M$ an {\it imaginary crystal $\mathbb A_0$-lattice} of $M$ if the following hold
\begin{enumerate}[(i)]
\item $\mathbb F(q^{1/2})\otimes_{\mathbb A_0}\mathcal L\cong M$\label{lattice1},
\item $\mathcal L=\oplus_{\lambda \in \pi}\mathcal L_\lambda$ and $\mathcal L_\lambda =\mathcal L\cap M_\lambda$,
\item $\tilde\Omega_\psi(m)\mathcal L\subseteq \mathcal L$ and $\tilde x^-_m\mathcal L\subseteq \mathcal L$ for all $m\in\mathbb Z$.
\end{enumerate}

An {\it imaginary crystal basis} of a $U_q({\mathfrak g})$-module $M \in \mathcal O^q_{\text{red,im}}$ is a pair $(\mathcal L,\mathcal B)$  satisfying
\begin{enumerate}
\item $\mathcal L$ is an imaginary crystal lattice of $M$,
\item $\mathcal B$ is an $\mathbb F$-basis of $\mathcal L/q\mathcal L\cong \mathbb F\otimes_{\mathcal A_0}\mathcal L$,
\item $\mathcal B=\sqcup_{\mu \in \pi}\mathcal B_\mu$, where $\mathcal B_\mu =\mathcal B\cap (\mathcal L_\mu/q\mathcal L_\mu)$, 
\item $\tilde x_m^-\mathcal B\subset\pm \mathcal B\cup\{0\}$ and $\tilde \Omega_\psi(m)\mathcal B\subset  \pm\mathcal B\cup \{0\}$,
\item For $m\in\mathbb Z$, if $ \tilde \Omega_\psi(-m)b\neq 0$ and $\tilde x_m^-b\neq0$ for $b\in \mathcal B$, then $\tilde x_m^-\tilde \Omega_\psi(-m)b=\tilde \Omega_\psi(-m)\tilde x_m^-b$.  .
\end{enumerate}

 For $\lambda\in \mathfrak h^*_{red}$,  $\bar M_q(\lambda) \in \mathcal O^q_{\text{red,im}}$, define
\begin{align*}
\mathcal L(\lambda):&=\sum_{k\geq 0\atop  i_1\geq \dots\geq i_k, i_j\in\mathbb Z}\mathbb A_0x^-_{i_1}\cdots x^-_{i_k}v_\lambda\subset \mathcal N_q^-v_\lambda=\bar M_q(\lambda).
\end{align*}

Then  $\mathcal L(\lambda)$ is a imaginary crystal $\mathbb A_0$-lattice \cite{CFM17} since
\begin{enumerate}[(i)]
\item $\mathbb F(q^{1/2})\otimes_{\mathbb A_0}\mathcal L(\lambda)\cong \bar M_q(\lambda)$\label{lattice1},
\item $\mathcal L(\lambda)=\oplus_{\mu \in \pi}\mathcal L(\lambda)_\mu$ where $\mathcal L(\lambda)_\mu =\mathcal L(\lambda)\cap  \bar M_q(\lambda)_\mu$,
\item $\tilde\Omega_\psi(m)\mathcal L(\lambda)\subseteq \mathcal L(\lambda)$ and $\tilde x^-_m\mathcal L(\lambda)\subseteq \mathcal L(\lambda)$ for all $m\in\mathbb Z$.
\end{enumerate}

\begin{prop}  \cite{CFM17} For $\lambda\in \mathfrak h^*_{red}$ we have
$$
\mathcal L(\lambda)=\left\{u\in \bar M_q(\lambda)\,|\, (u,\bar M_q(\lambda))\subset \mathbb A_0\right\}.
$$
\end{prop}

For $\lambda \in\mathfrak h^*$ define 
$$
\mathcal B(\lambda):=\left\{\tilde x_{i_1}^-\cdots \tilde x_{i_k}^-v_\lambda +q\mathcal L(\lambda)\in \mathcal L(\lambda)/q\mathcal L(\lambda)\,|\, i_1\geq \cdots \geq  i_k\right\}.
$$

\begin{thm} \cite{CFM17} For $\lambda\in {\mathfrak h}^*_{red}$, the pair $(\mathcal L(\lambda),\mathcal B(\lambda))$ is an imaginary crystal basis of the reduced imaginary Verma module $\bar M_q(\lambda)$.  
\end{thm}

Now suppose $M \in \mathcal O^q_{\text{red,im}}$. Then there exists $\lambda_i\in \mathfrak h^*_{red}$, $i\in I$,  with $M\cong \bigoplus_{i\in I}\bar M_q(\lambda_i)$. Let $(\mathcal L(\lambda_i),\mathcal B(\lambda_i))$ be the imaginary crystal basis for $\bar M_q(\lambda_i)$
for $i \in I$. Set $\mathcal L = \bigoplus_{i\in I}\mathcal L(\lambda_i)$ and  $\mathcal B = \bigsqcup_{i \in I}\mathcal B(\lambda_i)$.

\begin{thm} The pair $(\mathcal L,\mathcal B)$ is an imaginary crystal basis for $M \in \mathcal O^q_{\text{red,im}}$.  
\end{thm}
\begin{proof}

First let us see that $\mathcal L$ is an imaginary crystal lattice:
\begin{enumerate}[(i)]
\item $\mathbb F(q^{1/2})\otimes_{\mathbb A_0}\mathcal L\cong \oplus_{i\in I}\mathbb F(q^{1/2})\otimes_{\mathbb A_0}\mathcal L(\lambda_i)\cong \oplus_{i\in I}\bar M_q(\lambda_i) =M$,
\item First we show that $\mathcal L_\mu =\left(\oplus_{i\in I}\mathcal L(\lambda_i)\right)_\mu=\oplus_{i\in I}\mathcal L(\lambda_i)_\mu$
where $\mathcal L(\lambda_i)_\mu=\mathcal L(\lambda_i)\cap M_\mu$.
This follows because if $u\in \left(\oplus_{i\in I}\mathcal L(\lambda_i)\right)_\mu$, then 
$$
u=\sum_{i\in I}u_i
$$
with $u_i\in \mathcal L(\lambda_i)$ and $Ku= q^{\mu (h)}u$.  Moreover since $\mathcal L(\lambda_i) =\oplus_{\mu_i\in \pi}
\mathcal L(\lambda_i)_{\mu_i}$,
$$
u_i=\sum_{\mu_i\in\pi}u_{i,\mu_i}
$$
with $u_{i,\mu_i}\in \mathcal L(\lambda_i)$ and 
$$
\sum_{i\in I}\sum_{\mu_i\in\pi}q^{\mu(h)}u_{i,\mu_i}=\sum_{i\in I}q^{\mu(h)}u_i=q^{\mu(h)}u=Ku=\sum_{i\in I}Ku_i
$$
$$
=\sum_{i\in I}\sum_{\mu_i\in\pi}Ku_{i,\mu_i}= \sum_{i\in I}\sum_{\mu_i\in\pi}q^{\mu_i(h)}u_{i,\mu_i}. 
$$
Now since the sum $\oplus_{i\in I}\mathcal L(\lambda_i)$ is direct, the above gives us 
$$
 \sum_{\mu_i\in\pi}\left(q^{\mu(h)}-q^{\mu_i(h)}\right)u_{i,\mu_i} =0
$$
for each $i$.  Since $\mathcal L(\lambda_i)=\oplus_{\mu_i\in \pi}\mathcal L(\lambda_i)_{\mu_i}$ is a direct sum we have $\mu=\mu_i$ for all $i$.   Thus $u_{i,\mu_i}\in \mathcal L(\lambda_i)\cap M_\mu=\mathcal L(\lambda_i)_\mu$. 
Finally we have 
$\mathcal L=\oplus_{i\in I}\mathcal L(\lambda_i)=\oplus_{i\in I}\oplus_{\mu \in \pi}\mathcal L(\lambda_i)_\mu=\oplus_{\mu \in \pi}\oplus_{i\in I}\mathcal L(\lambda_i)_\mu=\oplus_{\mu \in \pi}\mathcal L_\mu$ and $\mathcal L_\mu =\mathcal L\cap M_\mu$.

\item $\tilde\Omega_\psi(m)\mathcal L=\tilde\Omega_\psi(m)\left(\oplus_{i\in I}\mathcal L(\lambda_i)\right)=\left(\oplus_{i\in I} \tilde\Omega_\psi(m)\mathcal L(\lambda_i)\right)\subseteq \oplus_{i\in I}\mathcal L(\lambda_i)=\mathcal L$ and $\tilde x^-_m\mathcal L= x^-_m\left(\oplus_{i\in I}\mathcal L(\lambda_i)\right)= \left(\oplus_{i\in I}\tilde x^-_m\mathcal L(\lambda_i)\right) \subseteq \oplus_{i\in I}\mathcal L(\lambda_i)=\mathcal L$ for all $m\in\mathbb Z$.
\end{enumerate}

This proves that $\mathcal L$ is an imaginary crystal lattice. 

We know that $\mathcal B(\lambda_i)$ is a $\mathbb F$-basis of $\mathcal L(\lambda_i)/q\mathcal L(\lambda_i) =\mathbb F\otimes  _{\mathcal A_0}\mathcal L(\lambda_i)$ for each $i\in I$.
Now 
\begin{align*}
\mathcal L/q\mathcal L&=\left(\oplus_{i\in I}\mathcal L(\lambda_i)\right)/q\left(\oplus_{i\in I}\mathcal L(\lambda_i)\right) \\
&\cong  \oplus_{i\in I}\left(\mathcal L(\lambda_i)/q\mathcal L(\lambda_i)\right) \cong \oplus_{i\in I}\left(\mathbb F\otimes_{\mathcal A_0} \mathcal L(\lambda_i)\right)\
\cong \mathbb F\otimes_{\mathcal A_0} \mathcal L
\end{align*}
since $\mathcal L=\oplus_{i\in I}\mathcal L(\lambda_i)$.
Hence $\mathcal L/q\mathcal L$ has the $\mathbb F$-basis $\sqcup_{i\in I}\mathcal B(\lambda_i)=\mathcal B$.

For (3) we have 
$$
\mathcal B=\sqcup_{i\in I} \mathcal B(\lambda_i)=\sqcup_{i\in I}\sqcup_{\mu \in \pi} \mathcal B(\lambda_i)_\mu=\sqcup_{\mu \in \pi} \sqcup_{i\in I}\mathcal B(\lambda_i)_\mu=\sqcup_{\mu \in \pi} \mathcal B_\mu
$$
where $\mathcal B_\mu:=\sqcup_{i\in I}\mathcal B(\lambda_i)_\mu$ and 
\begin{align*}
\mathcal B_\mu&=\sqcup_{i\in I}\left(\mathcal B(\lambda_i)\cap(\mathcal L(\lambda_i)_\mu /q\mathcal L(\lambda_i)_\mu\right) \\
&=\sqcup_{i\in I}\mathcal B(\lambda_i)\cap\left(\oplus_{j \in I}(\mathcal L(\lambda_j)_\mu /q\mathcal L(\lambda_j)_\mu)\right)\\
&=\left(\sqcup_{j\in I}\mathcal B(\lambda_i)\right)\cap\mathcal L_\mu /q\mathcal L_\mu \\
&= \mathcal B \cap\mathcal L_\mu /q\mathcal L_\mu.
\end{align*}

For (4) one notes that  $\tilde x_m^-\mathcal B=\sqcup_{i\in I}\tilde x_m^-\mathcal B(\lambda_i)\subseteq \pm \sqcup_{i\in I} \mathcal B(\lambda_i)\cup \{0\}\subset\pm \mathcal B\cup\{0\}$ and similarly $\tilde \Omega_\psi(m)\mathcal B\subset  \pm\mathcal B\cup \{0\}$,

For (5) we take $b\in \mathcal B$ and suppose $ \tilde \Omega_\psi(-m)b\neq 0$ and $\tilde x_m^-b\neq0$.  Now $b\in \mathcal B(\lambda_i)$ for some $i \in I$. Since $(\mathcal L(\lambda_i),\mathcal B(\lambda_i))$ is an imaginary crystal basis, we have $\tilde x_m^-\tilde \Omega_\psi(-m)b=\tilde \Omega_\psi(-m)\tilde x_m^-b$.

\end{proof}

We also have the following partial converse.
\begin{thm}
Let $M=M_1\oplus M_2$ where $M_1$ and $M_2$ are modules in the category $\mathcal O_{\text{red},\text{im}}^q$ and suppose $(\mathcal L,\mathcal B)$ is an imaginary crystal basis for $M$.  Furthermore, suppose that there exists $\mathcal A_0$-submodules $\mathcal L_j\subset M_j$, and subsets $\mathcal B_j\subset \mathcal L_j/q\mathcal L_j, \, j =1,2$ such that $\mathcal L=\mathcal L_1\oplus \mathcal L_2$ and $\mathcal B=\mathcal B_1\sqcup \mathcal B_2$.  Then $(\mathcal L_j,\mathcal B_j)$ is an imaginary crystal basis of $M_j, \, j=1,2$. 
\end{thm}
\begin{proof}
It is straightforward to see that $\mathbb F(q^{1/2})\otimes _{\mathcal A_0}(\mathcal L_j)_\mu\cong (M_j)_\mu$ ($\mu\in \pi$),  $\mathcal L_j=\mathcal L\cap M_j$ and $\mathcal B_j=\mathcal B\cap (\mathcal L_j/q\mathcal L_j)$ for $j=1,2$ (see for instance \cite[Theorem 4.2.10]{HK}).

Let $u\in \mathcal L_\mu,$ for some $\mu \in \pi$. Then  $u=u_1+u_2 \in \mathcal L_1\oplus \mathcal L_2$ and 
$$
Ku_1+Ku_2=Ku=q^{\mu(h)}=q^{\mu (h)}u_1+q^{\mu (h)}u_2
$$
Then $Ku_1-q^{\mu(h)}u_1=-Ku_2+q^{\mu(h)}u_2$ which must be zero since $\mathcal L\cap \mathcal L_2=\{0\}$.  Thus $u_j\in \mathcal L_{j,\mu}, \, j = 1,2$. Hence $\mathcal L_\mu=\mathcal L_{1,\mu}\oplus \mathcal L_{2,\mu}$. 

 For $u\in \mathcal L_j \subset \mathcal L$, $u\in\oplus_{\mu\in \pi} \mathcal L_\mu$, so  as a consequence $u=\sum_{\mu\in \pi}u_\mu $ with $u_\mu \in \mathcal L_\mu$ and we can write $u_\mu =u_{1,\mu}+u_{2,\mu}$ with $u_{j,\mu }\in (\mathcal L_j)_\mu:= \mathcal L_j\cap M_\mu$.  Consequently $u-\sum_{\mu\in\pi}u_{j,\mu}=\sum_{\lambda \in \pi}u_{k,\mu}\in\mathcal L_j\cap \mathcal L_k$ with $k\neq j$.  Hence  $u-\sum_{\mu\in\pi}u_{j,\mu}=0$ and we have $u\in \oplus_{\mu\in \pi}\mathcal L_{j,\mu}$.

 Let $p_j: M\to M_j$ denote the natural projections which are $U_q(\mathfrak g)$-module homomorphisms . Consider $u_j\in\mathcal L_j$. Since $\Omega_\psi(m)\mathcal L\subset \mathcal L$ and $\bar x_m^-\mathcal L\subset \mathcal L$ we write $\Omega_\psi(m)u_1=\bar u_1+\bar u_2$ and $\bar x_m^-u_1=\check u_1+\check u_2$ with $\check u_j,\bar u_j\in \mathcal L_j$. By Theorem \ref{commute} we have 
 $$
 \Omega_\psi(m)(u_1)= \Omega_\psi(m) p_1(u_1)= p_1\Omega_\psi(m)(u_1)= \bar u_1
 $$
 and 
 $$
\bar x_m^-(u_1)=\bar x_m^- p_1(u_1)= p_1\bar x_m^-(u_1)= \check u_1
 $$
Hence $\Omega_\psi(m)(u_1)\in \mathcal L_1$ and $\bar x_m^-(u_1)\in \mathcal L_1$.  Similarly $\Omega_\psi(m)(\mathcal L_2)\subset \mathcal L_2$ and $\bar x_m^-(\mathcal L_2)\subset \mathcal L_2$. 
This concludes the proof that $\mathcal L_j$ are imaginary crystal lattices.

We can write 
$$
\mathcal L_1/q\mathcal L_1\oplus\mathcal L_2/q\mathcal L_2\cong \mathcal L/q\mathcal L\cong \mathbb F\otimes_{\mathcal A_0}\mathcal L\cong  \mathbb F\otimes_{\mathcal A_0}\mathcal L_1\oplus \mathbb F\otimes_{\mathcal A_0}\mathcal L_2 
$$
Using this isomorphism we have $\mathcal B_j=\mathcal B\cap (\mathcal L_j/q\mathcal L_j)$ is an $\mathbb F$-basis of $\mathcal L_j/q\mathcal L_j\cong \mathbb F\otimes_{\mathcal A_0}\mathcal L_j$.
Next we have 
$$
\mathcal B=\mathcal B_1\sqcup\mathcal B_2
$$
and thus $\mathcal B_j=\sqcup_{\mu\in \pi}(\mathcal B_j)_\mu$ where $(\mathcal B_j)_\mu=\mathcal B_j\cap ((\mathcal L_j)_\mu/q(\mathcal L_j)_\mu)$. 

Since the operators $\Omega_\psi(m)$ and $\bar x_m^-$ commute with $U_q(\mathfrak g)$-module homomorphisms in the category $\mathcal O_{\text{red},\text{im}}^q$ (Theorem \ref{commute}), we have that $\Omega_\psi(m)\mathcal B_j\subset \mathcal B_j\cup\{0\}$ and $\bar x_m^-\mathcal B_j\subset \mathcal B_j\cup\{0\}$ for all $m$ and $j=1,2$. 

For $m\in\mathbb Z$, if $ \tilde \Omega_\psi(-m)b\neq 0$ and $\tilde x_m^-b\neq0$ for $b\in \mathcal B_j$ (and hence $b\in \mathcal B$), then $\tilde x_m^-\tilde \Omega_\psi(-m)b=\tilde \Omega_\psi(-m)\tilde x_m^-b$.

This completes the proof that $(\mathcal L_j,\mathcal B_j)$ is an imaginary crystal basis for $j = 1,2$.
 \end{proof}

\section*{Acknowledgement}
The first author was partially support by a Simons Foundation Grant \#319261. The second author was supported in part by the CNPq grant \#301320/2013-6 and by the FAPESP grant \#2014/09310-5. The third author was partially support by the Simons Foundation Grant \#307555.

\bibliography{math}

\def\cprime{$'$}
\providecommand{\bysame}{\leavevmode\hbox to3em{\hrulefill}\thinspace}
\providecommand{\MR}{\relax\ifhmode\unskip\space\fi MR }
\providecommand{\MRhref}[2]{%
  \href{http://www.ams.org/mathscinet-getitem?mr=#1}{#2}
}
\providecommand{\href}[2]{#2}
\bibliographystyle{amsalpha}

\end{document}